\documentclass[11pt]{amsart}
\usepackage{graphicx,amsmath, amssymb}
\usepackage{color,breqn}
\input xy

\xyoption{all}

\begin{document}

\newtheorem{defn}{Definition}

\newcommand{\ad}{\ensuremath{\operatorname{ad}}}
\newcommand{\Aut}{\ensuremath{\operatorname{Aut}}}
\newcommand{\Vol}{\ensuremath{\operatorname{Vol}}}
\newcommand{\ch}{\ensuremath{\operatorname{char}}}
\newcommand{\Gom}{\ensuremath{\operatorname{Gom}}}
\newcommand{\Gal}{\ensuremath{\operatorname{Gal}}}
\newcommand{\Stab}{\ensuremath{\operatorname{Stab}}}
\newcommand{\Ends}{\ensuremath{\operatorname{Ends}}}
\newcommand{\rank}{\ensuremath{\operatorname{rank}}}
\newcommand{\dist}{\ensuremath{\operatorname{dist}}}
\newcommand{\Ball}{\ensuremath{\operatorname{Ball}}}
\newcommand{\Br}{\ensuremath{\operatorname{Br}}}
\newcommand{\id}{\ensuremath{\operatorname{id}}}
\newcommand{\Ord}{\operatorname{Ord}}
\newcommand{\End}{\operatorname{End}}
\newcommand{\F}{\mathbb{F}}
\newcommand{\Fqt}{\F_q(\!(t)\!)\!}
\newcommand{\N}{\mathbb{N}}
\newcommand{\G}{\Gamma}
\newcommand{\bs}{\backslash}
\newcommand{\forget}[1]{}
\newcommand{\quot}{\bs \! \bs}
\newcommand{\cA}{\mathcal{A}}
\newcommand{\cE}{\mathcal{E}}
\newcommand{\cG}{\mathcal{G}}
\newcommand{\cM}{\mathcal{M}}
\newcommand{\cU}{\mathcal{U}}
\newcommand{\la}{\langle}
\newcommand{\ra}{\rangle}
\newcommand{\bX}{\partial X}
\newcommand{\ep}{\varepsilon}

\newcommand{\Atwotilde}{\widetilde{A}_2}
\newcommand{\Antilde}{\widetilde{A}_n}
\newcommand{\Nitilde}{\widetilde{N}_i}
\newcommand{\Nzerotilde}{\widetilde{N}_0}
\newcommand{\Nonetilde}{\widetilde{N}_1}
\newcommand{\Ntwotilde}{\widetilde{N}_2}
\newcommand{\Ndminusonetilde}{\widetilde{N}_{d-1}}
\newcommand{\Sitilde}{\widetilde{S}_i}
\newcommand{\Szerotilde}{\widetilde{S}_0}
\newcommand{\Sonetilde}{\widetilde{S}_1}
\newcommand{\Stwotilde}{\widetilde{S}_2}
\newcommand{\tG}{\widetilde\Gamma}

\newcommand{\SL}{\ensuremath{\operatorname{SL}}}
\newcommand{\GL}{\ensuremath{\operatorname{GL}}}
\newcommand{\PSL}{\ensuremath{\operatorname{PSL}}}
\newcommand{\PGL}{\ensuremath{\operatorname{PGL}}}

\newcommand{\PG}{\ensuremath{\operatorname{PG}}}

\newcommand{\mK}{{\mathbb K}}
\newcommand{\mE}{{\mathbb E}}
\newcommand{\mN}{{\mathbb N}}
\newcommand{\mC}{{\mathbb C}}
\newcommand{\mB}{{\mathbb B}}
\newcommand{\mF}{{\mathbb F}}
\newcommand{\mP}{\mathbb{P}}
\newcommand{\mZ}{\mathbb{Z}}

\newcommand{\Fq}{\mF_{q}}
\newcommand{\Fqd}{\mF_{q^d}}

\newcommand{\cB}{{\mathcal B}}
\newcommand{\cO}{{\mathcal O}}

\newcommand{\oP}{{\overline{\Psi}}}


\newenvironment{mylist}{\begin{list}{}{
\setlength{\itemsep}{0mm}
\setlength{\parskip}{0mm}
\setlength{\topsep}{1mm}
\setlength{\parsep}{0mm}
\setlength{\itemsep}{0mm}
\setlength{\labelwidth}{6mm}
\setlength{\labelsep}{3mm}
\setlength{\itemindent}{0mm}
\setlength{\leftmargin}{9mm}
\setlength{\listparindent}{6mm}
}}{\end{list}}


\newtheorem{theorem}{Theorem}
\newtheorem*{theorem1'}{Theorem 1$'$}
\newtheorem{prop}[theorem]{Proposition}
\newtheorem{lemma}[theorem]{Lemma}
\newtheorem{corollary}[theorem]{Corollary}
\newtheorem{question}[theorem]{Question}
\newtheorem{remark}[theorem]{Remark}
\newtheorem{definition}[theorem]{Definition}
\newtheorem{conjecture}[theorem]{Conjecture}

\title{Cocompact lattices on $\tilde{A}_n$ buildings}

\author{Inna Capdeboscq} \address{Mathematics Institute, Zeeman Building, University of Warwick, Coventry
CV4 7AL, UK} 
\email{I.Korchagina@warwick.ac.uk}
\thanks{This research of the third author was supported by EPSRC Grant No. EP/D073626/2 and is now supported in part by ARC Grant No. DP110100440.  Thomas is also supported in part by an Australian Postdoctoral Fellowship.}

\author{Dmitriy Rumynin} \address{Mathematics Institute, Zeeman Building, University of Warwick, Coventry
CV4 7AL, UK} \email{D.Rumynin@warwick.ac.uk}

\author{Anne Thomas}\address{School of Mathematics and Statistics F07, University of Sydney NSW 2006, Australia}
\email{anne.thomas@sydney.edu.au}

\date{\today}

\begin{abstract}  We construct cocompact lattices $\G'_0 < \G_0$ in the group $G = \PGL_d(\Fqt)$ which are type-preserving and act transitively on the set of vertices of each type in the building $\Delta$ associated to $G$.  The stabiliser of each vertex in $\G'_0$ is a Singer cycle and  the stabiliser of each vertex in $\G_0$ is isomorphic to the normaliser of a Singer cycle in $\PGL_d(q)$.  We then show that the intersections of $\G'_0$ and $\G_0$ with $\PSL_d(\Fqt)$ are lattices in $\PSL_d(\Fqt)$, and identify the pairs $(d,q)$ such that the entire lattice $\G'_0$ or $\G_0$ is contained in $\PSL_d(\Fqt)$.  Finally we discuss minimality of covolumes of cocompact lattices in $\SL_3(\Fqt)$.  Our proofs combine a construction of Cartwright--Steger \cite{CS} with results about Singer cycles and their normalisers, and geometric arguments.
\end{abstract}

\maketitle

\section{Introduction}\label{s:intro}

Let $\F_q$ be the finite field of order $q$ where $q$ is a power of a prime $p$, and let $K$ be the field $\Fqt$ of formal Laurent series over $\F_q$, with discrete valuation $\nu :K^\times \to \mathbb{Z}$.  Let $\Delta$ be the building $\tilde{A}_n(K,\nu)$, as constructed in, for example \cite[Chapter 9]{Ronan} (see also Section \ref{s:buildings} below).  Then $\Delta$ is an affine building of type $\tilde{A}_n$, meaning that the apartments of $\Delta$ are isometric images of the Coxeter complex of type $\tilde{A}_n$.  The link of each vertex of $\Delta$ may be identified with the $n$--dimensional projective space $\PG(n,q)$ over $\F_q$. 

Let $d = n+1$ and let $G$ be the group $G = \cG(K)$, where $\cG$ is in the set $\{ \GL_d, \PGL_d, \SL_d, \PSL_d \}$.  Then $G$ is a totally disconnected, locally compact group which acts on $\Delta$ with kernel $Z(G)$.  It follows from a theorem of Tits \cite{Ti2} that $G/Z(G)$ is cocompact in the full automorphism group of $\Delta$.   If $\cG$ is $\GL_d$ or $\PGL_d$, then the $G$--action is type-rotating and transitive on the vertex set of $\Delta$, while if $\cG$ is $\SL_d$ or $\PSL_d$, then the $G$--action is type-preserving and transitive on each type of vertex.  See Section \ref{s:preliminaries} below for definitions of these terms.

By definition, a subgroup $\G \leq G$ is a \emph{lattice} if it is a discrete subgroup such that $\G \bs G$ admits a finite $G$--invariant measure, and a lattice $\G$ is \emph{cocompact} if $\G \bs G$ is compact.  In the cases $\cG = \PGL_d$, $\SL_d$ and $\PSL_d$, the centre of $G = \cG(K)$ is compact, hence $G$ acts on $\Delta$ with compact vertex stabilisers.  A subgroup $\G \leq G$ is then discrete if and only if $\G$ acts on $\Delta$ with finite vertex stabilisers, and if $\G \leq G$ is discrete then $\G$ is a cocompact lattice if and only if, in addition, $\Gamma$ acts cocompactly on $\Delta$.  Given any lattice $\G$ and a set $A$ of vertices of $\Delta$ which represent the orbits of $\G$, the Haar measure $\mu$ on $G$ may be normalised so that~$\mu(\G \bs G)$, the covolume of $\G$ in $G$, is given by the series $\sum_{a \in A} |\Stab_\G(a)|^{-1}$ (see \cite{BL}).  This is a finite sum if and only if $\G$ is cocompact.

The existence of an arithmetic cocompact lattice in $G = \cG(K)$ is due to Borel--Harder \cite{BH}.  By Margulis' Arithmeticity Theorem \cite{M}, if $d \geq 3$ then every lattice in such $G$ is arithmetic.  In the rank~$1$ case, that is, for $d = 2$, the building $\Delta$ is a tree of valence $q+1$, and there are several additional known constructions of cocompact lattices in $G$.  For example, Fig\'a-Talamanca and Nebbia \cite{FTN} constructed lattices in $G=\PGL_2(\Fqt)$ which act simply transitively on the set of vertices of the tree $\Delta$.  Such lattices are necessarily free products of $s$ copies of the cyclic group of order $2$, and $t$ copies of the infinite cyclic group, where $s + t = q+1$.  The cocompact lattices of minimal covolume in $G = \SL_2(\Fqt)$ were constructed in \cite{L1,LW}.  These lattices are fundamental groups of finite graphs of finite groups which, using Bass' covering theory for graphs of groups \cite{B}, are embedded in $G$.  Lubotzky \cite{L2} also constructed a moduli space of cocompact lattices in $\SL_2(\Fqt)$ which are finitely generated free groups, using a Schottky-type construction.  

If $d = 3$, then additional constructions of lattices in $G$ may be complicated by the fact that there exist uncountably many ``exotic" $\tilde{A}_2$--buildings, that is, buildings of type $\tilde{A}_2$ which are not of the form $\tilde{A}_2(K,\nu)$ for any field $K$, not necessarily commutative, with discrete valuation $\nu$ (Tits \cite{Ti1}).  On the other hand for $d \geq 4$, that is, for $n \geq 3$, there are no exotic building of type $\tilde{A}_{n}$ (Tits \cite{Ti3}).  

For $d \geq 3$, there exists a chamber-transitive lattice in $\PSL_d(\Fqt)$ if and only if $d = 3$ and $q = 2$ or $q = 8$ (see \cite{KLT} and its references).   Lattices in the group $G = \PGL_d(\Fqt)$ which act simply transitively on the vertex set of the associated building $\Delta$ were constructed for the case $d = 3$ in \cite{CMSZ1}, and for $d > 3$ in \cite{CS}.  We will describe the work of \cite{CMSZ1} and \cite{CS} further below.  In addition, in the case $d = 3$, Ronan \cite{R} constructed lattices acting simply transitively on the set of vertices of the same type in some, possibly exotic, $\tilde{A}_2$--building, and Essert \cite{E} constructed lattices acting simply transitively on the set of panels of the same type in some, again possibly exotic, $\tilde{A}_2$--building.   Essert's construction used complexes of groups (see \cite{BrH}), and had vertex stabilisers cyclic groups acting simply transitively on the set of points and lines of $\PG(2,q)$, the projective plane over $\F_q$.  Our work resolves some open questions of \cite{E}, as we explain below.

Our main results are Theorems \ref{t:PGL} and \ref{t:PSL} below.  See Section \ref{s:singer} below for the definition of a Singer cycle in $\PGL_d(q)$; such a group acts simply transitively on the set of points and lines of $\PG(2,q)$.  We first construct lattices in $\PGL_d(\Fqt)$.

\begin{theorem}\label{t:PGL}  Let $G = \PGL_d(\Fqt)$ and let $\Delta$ be the building associated to $G$.  Then $G$ admits cocompact lattices $\G_0' \leq \G_0$ such that:
\begin{itemize} 
\item the action of $\G_0'$ and of $\G_0$ on $\Delta$ is type-preserving and transitive on each type of vertex;
\item the stabiliser of each vertex in $\G_0'$ is isomorphic to a Singer cycle in $\PGL_d(q)$; and
\item the stabiliser of each vertex in $\G_0$ is isomorphic to the normaliser of a Singer cycle in $\PGL_d(q)$.
\end{itemize}
Moreover $\G_0'$ and $\G_0$ are generated by their $d$ subgroups which are the stabilisers of the vertices of the standard chamber in $\Delta$.
\end{theorem}

\noindent In fact, the  stabiliser of each vertex in $\G_0'$ is always contained in a finite subgroup of $G$ isomorphic to $\PGL_d(q)$. However for the vertex stabilisers of $\G_0$ the situation is trickier. If $(p,d)=1$,  then 
 the  stabiliser of each vertex in $\G_0$ is indeed contained in a finite subgroup of $G$ isomorphic to $\PGL_d(q)$.  On the other hand, as we discuss in Section  \ref{s:construction}, if $p$ divides $d$, then the  stabiliser of each vertex in $\G_0$ intersects  a finite subgroup of $G$ isomorphic to $\PGL_d(q)$ in a subgroup of index $p^a$, where $d = p^a b$ and $(p,b) = 1$.

We then construct lattices in $\PSL_d(\Fqt)$, where we identify the group $\PSL_d(\Fqt)$ with a subgroup of $\PGL_d(\Fqt)$.  Our notation continues from Theorem \ref{t:PGL}.

\begin{theorem}\label{t:PSL}  The groups $$\Lambda_0':=\G_0' \cap \PSL_d(\Fqt) \quad \mbox{and} \quad \Lambda_0:=\G_0 \cap \PSL_d(\Fqt)$$ are cocompact lattices in $\PSL_d(\Fqt)$, necessarily type-preserving.  Moreover: 
\begin{enumerate}
\item  Suppose that $(d,q-1) = 1$.
\begin{enumerate}
\item If $p$ does not divide $d$, then $\Lambda_0' = \G_0'$ and $\Lambda_0 = \G_0$.
\item If $p$ divides $d$, then $\Lambda_0' = \G_0'$ and $\Lambda_0$ is a proper subgroup of $\G_0$.
 \end{enumerate}
\item If $(d,q-1) \neq 1$, then $\Lambda_0'$ is a proper subgroup of $\G_0'$ and $\Lambda_0$ is a proper subgroup of $\G_0$.
\end{enumerate}
In all cases where $\Lambda_0' = \G_0'$ (respectively, $\Lambda_0 = \G_0$), it follows that $\G_0'$ (respectively, $\G_0$) is a cocompact lattice in $\PSL_d(\Fqt)$ with properties as described in Theorem \ref{t:PGL}.  
\end{theorem}

\noindent In particular, in Section \ref{s:d>3} we give the precise structure of the vertex stabilisers in $\Lambda_0$ and $\Lambda_0'$, and we describe the cases in which these lattices can be generated by their vertex stabilisers.

Since the centre of $\SL_d(\Fqt)$ is finite and fixes $\Delta$ pointwise, if $\G$ is any lattice in $\PSL_d(\Fqt)$ then the full pre-image of $\G$ under the canonical epimorphism is a cocompact lattice in $\SL_d(\Fqt)$.  We thus obtain lattices in $\SL_d(\Fqt)$ as well.  Of course if $(d,q-1) = 1$, then the centre of $\SL_d(\Fqt)$ is trivial, and so, for example, $\G'_0 = \Lambda_0'$ itself is a lattice in $\SL_d(\Fqt)$.

Our original motivation was to find cocompact lattices of minimal covolume in $\SL_3(\Fqt)$.  For this, it was natural to consider vertex stabilisers which are Singer cycles or normalisers of Singer cycles, since these are the vertex stabilisers of the cocompact lattices of minimal covolume in $\SL_2(\Fqt)$ (see \cite{L1, LW}) and more generally in topological rank $2$ Kac--Moody groups $G$ over $\F_q$ (see \cite{CT}, where the minimality result holds under the conjecture that cocompact lattices in such $G$ do not contain $p$--elements).  In Section \ref{s:no p elements} below, we show that a lattice $\G < \SL_d(\Fqt)$ is cocompact if and only if it does not contain any $p$--elements.  This analogue of Godement's Compactness Criterion will not surprise experts, but we were not able to find it in the literature.   In Section \ref{s:minimality}, we are able to use this criterion  to show that when $(3,q-1) = 1$ and $p = 2$, the lattice $\G_0$ is a cocompact lattice in $\SL_3(\Fqt)$ of minimal covolume.  We also show that when $(3,q-1) = 1$ and $p = 3$, $\G_0'$ is a maximal lattice in $\SL_3(\Fqt)$, and that when $(3,q-1) = 1$ and $p\neq 3$,
 $\G_0$ is a maximal lattice in $\SL_3(\Fqt)$.  We conclude the discussion of covolumes with a conjecture about the cocompact lattice of minimal covolume in $\SL_3(\Fqt)$ when $(3,q-1) = 1$ and $p$ is odd.  

Finally, in Section \ref{s:essert}, we discuss how our results answer some open questions from the work of Essert \cite{E}.    For example, Theorem~\ref{t:PSL} implies that for all $q$ such that $(3,q-1) =1$, the group $\SL_3(\Fqt)$ contains a lattice which acts simply transitively on the set of panels of each type in $\Delta$. 

To obtain the lattices $\G'_0$ and $\G_0$ in Theorem \ref{t:PGL}, we use a construction of Cartwright and Steger from \cite{CS}, which generalises work of \cite{CMSZ1}.  This construction gives cocompact lattices $\G < \tG$ in the automorphism group $\Aut(\widetilde{\mathcal{A}})$ of a certain algebra $\widetilde{\mathcal{A}}$, such that $\Aut(\widetilde{\mathcal{A}})$ is isomorphic to $\PGL_d(\Fqt)$.
The lattice $\G$ 
acts simply transitively on the vertex set of $\Delta$, and $\tG = H\G$ 
where $H$ is a finite group which is the stabiliser in $\tG$ of a vertex of $\Delta$.  
We review and slightly extend this construction in Section~\ref{s:CS}. 
Our treatment applies to any cyclic Galois extension rather than just the extension of finite fields $\Fqd \supseteq \Fq$.  
In Section~\ref{s:applications} we choose an explicit isomorphism $\Aut(\widetilde{\mathcal{A}})\rightarrow \PGL_d(\Fqt)$ and so 
move our discussion explicitly into $\PGL_d(\Fqt)$.
We also show that $H$ is  isomorphic to the normaliser of a Singer cycle $S$ in $\PGL_d(q)$.

For expository reasons, we then divide the remaining proof of Theorems \ref{t:PGL} and \ref{t:PSL} between the case $d = 3$, in Section \ref{s:d=3}, and the cases $d > 3$, in Section \ref{s:d>3}.  For all $d \geq 3$, we define $\G'_0$ and $\G_0$ to be the subgroups of $\tG$ generated by suitable $\tG$--conjugates of $S$ or $H$, respectively.  Since $\tG$ is a discrete subgroup of $\PGL_d(\Fqt)$, it is immediate that $\G'_0$ and $\G_0$ are discrete.  Using geometric arguments, we then show that $\G'_0$ and $\G_0$ act cocompactly on $\Delta$, hence are cocompact lattices.  The main additional ingredient in the proof of Theorem \ref{t:PSL} is our determination in Section \ref{s:CS} of the intersection of $H$ with $\PSL_3(\Fqt)$.  This intersection is also used to show that, for certain values of $d$ and $q$, in fact $\G_0 = \tG \cap \PSL_d(\Fqt)$ or $\G'_0 = \tG \cap \PSL_d(\Fqt)$.

\subsection*{Acknowledgements}  We would like to thank Alina Vdovina and Kevin Wortman for helpful conversations, and the University of Sydney  and the University of Warwick for travel support.  

\section{Preliminaries}\label{s:preliminaries}

We briefly recall some definitions and results, and fix notation.

\subsection{Singer cycles and projective spaces}\label{s:singer}

The following definitions and results are taken from \cite{CdR}.  Let $q$ be a power of a prime $p$ and let $V$ be the vector space $\F_q^d$, for $d \geq 2$.  A cyclic subgroup $S$ of $\GL_d(q)$ that acts simply transitively on the set of non-zero vectors of $V$ is called a \emph{Singer cycle of $\GL_d(q)$}.  Its generator $s$ is an element of $\GL_d(q)$ of order $(q^d-1)$ and so $|S|=q^d-1$. The image of a Singer cycle of $\GL_d(q)$ in $\PGL_d(q)$ under the canonical epimorphism is called a \emph{Singer cycle of $\PGL_d(q)$}.   The intersection of a Singer cycle $S$ of $\GL_d(q)$ with $\SL_d(q)$, that is, $S\cap\SL_d(q)$, is called a \emph{Singer cycle of
 $\SL_d(q)$}. Its image under the canonical epimorphism from $\SL_d(q)$ onto $\PSL_d(q)$ is called a \emph{Singer cycle of $\PSL_d(q)$}.   A Singer cycle of $\PGL_d(q)$ or of $\SL_d(q)$ has order $\frac{q^d - 1}{q-1}$, and a Singer cycle of $\PSL_d(q)$ has order $\frac{q^d - 1}{(q-1)\delta}$ where $\delta  = (d,q-1)$.  

Note that a Singer cycle of $\PGL_d(q)$ acts simply transitively on the set of $1$--dimensional subspaces of $V$, and hence acts simply transitively on the set of $(d-1)$--dimensional subspaces of $V$ as well.  

We denote by $\PG(n,q)$ the projective space of dimension $n=d-1$ over the finite field $\Fq$.  Recall that the set of \emph{points} of $\PG(n,q)$ is the set of $1$--dimensional subspaces of $V$, and the set of \emph{lines} is the set of $2$--dimensional subspaces of $V$.  

Thus in particular, a Singer cycle of $\PGL_3(q)$ acts simply transitively on both the set of points and the set of lines of the projective plane $\PG(2,q)$.  
If $(3,q-1)=1$, the order of a Singer cycle of $\PSL_3(q)$,  $\frac{q^3-1}{q-1}$, coincides with the order of a Singer cycle of $\PGL_3(q)$. It follows immediately that in this case, if we identify $\PSL_3(q)$ with a subgroup of $\PGL_3(q)$, the Singer cycles of $\PSL_3(q)$ and $\PGL_3(q)$ coincide.  On the other hand, if $3$ divides $q-1$ (that is, $(3,q-1)=3\neq 1$), the order of a Singer cycle of $\PSL_3(q)$ is $\frac{q^3-1}{3(q-1)}$ and so this subgroup cannot act transitively on the $q^2+q+1$ points of the projective plane  
$\PG(2,q)$.  In fact, a simple application of Orbit-Stabiliser Theorem shows that even the normaliser of a Singer cycle of $\PSL_3(q)$ cannot act transitively on the points of $\PG(3,q)$.   
Moreover, for large enough $q$, the only $p'$--subgroups of $\PSL_3(q)$  that act transitively on the  points of $\PG(2,q)$ are Singer cycles and their normalisers and only when $(3,q-1)=1$. This follows immediately
from an inspection of the maximal  subgroups of $\PSL_3(q)$ that are provided by a result of  Hartley and Mitchell (Theorem 6.5.3 of \cite{GLS3}).  Hence for large enough $q$, if $3$ divides $(q-1)$ there are no $p'$--subgroups of $\SL_3(q)$ that act transitively on the set of points of $\PG(2,q)$.

\subsection{Buildings of type $\Antilde$}\label{s:buildings}

We assume basic knowledge of buildings, and extract from \cite{CMSZ1} and \cite{CS} the facts that we will need.  A reference for this theory is \cite{Ronan}.  We also recall the Levi decomposition of a vertex stabiliser in $\SL_d(\Fqt)$ or $\PSL_d(\Fqt)$.

Let $\Delta$ be the building $\tilde{A}_n(K,\nu)$ on which $\cG(K)$ acts, where $K = \Fq(\!(t)\!)$, as in the introduction.  Let $\mathcal{O} := \{ a \in K : \nu(a) \geq 0\} = \Fq[[t]]$.  A \emph{lattice in $K^d$} is a free $\mathcal{O}$--submodule of $K^d$ of rank $d$, and two lattices $L$ and $L'$ are said to be \emph{equivalent} if $L' = La$ for some $a \in K^\times$.  The vertices of $\Delta$ are the equivalence classes of lattices in $K^d$.  The group $G = \PGL_d(\Fqt)$ acts transitively on the vertex set of $\Delta$, so that the stabiliser of the equivalence class represented by $\mathcal{O}^d$ is $P_0:=\PGL_d(\Fq[[t]])$.  Thus we may identify the vertex set of $\Delta$ with the set of cosets $G/P_0$.  For $g \in \GL_d(\Fqt)$, we denote the image of $g$ in $\PGL_d(\Fqt)$ by $\overline{g}$.  The \emph{type} of the vertex $\overline{g}P_0$ is $\nu(\det(g)) \pmod d$.  

Let $v_0$ be the vertex of $\Delta$ identified with the trivial coset of $P_0$.  Then $v_0$ is the vertex of type $0$ in the standard chamber of $\Delta$.  For $i = 1,\ldots, d-1$, the vertex $v_i$ of type $i$ in the standard chamber is a coset of the form $\overline{g_i}P_0$ where $g_i \in \GL_d(\Fqt)$ has entries in $\mathcal{O}$, and $\nu(\det(g_i)) = i$.  The set of all vertices adjacent to $v_0$ corresponds to the elements of the projective space $\PG(n,q)$, and moreover we may choose the types so that for each $i = 1,\ldots, d-1$, the vertices neighbouring $v_0$ of type $i$ correspond to the $i$--dimensional subspaces of $V = \Fq^d$.    

The action of each $\overline{g} \in \PGL_d(\Fqt)$ on $\Delta$ induces a permutation of the set of types of the form $i \mapsto i + c \pmod d$, where $c = \nu(\det(g))$.  Any automorphism of $\Delta$ which induces a permutation of types of the form $i \mapsto i + c \pmod d$, for some $c$, is said to be \emph{type-rotating}.  In particular, a type-rotating automorphism fixes either no type or all types.

We will need the following decomposition of vertex stabilisers, which is a special case of a result for topological Kac--Moody groups in \cite{CR}.

\begin{prop}[Levi decomposition] \label{p:levi} Let $G = \cG(\Fqt)$ where $\cG$ is $\SL_d$ or $\PSL_d$, $d \geq 2$, and $q$ is a power of a prime $p$.  Let $v$ be a vertex of the building $\Delta$ associated to $G$.  Then the stabiliser of $v$ in $G$ has Levi decomposition
\[ L_v \ltimes U_v \]
where $L_v$ is isomorphic to the finite group $\cG(\F_q)$, and $U_v$ is pro--$p$.
\end{prop}

\section{Generalisation of Cartwright--Steger construction}\label{s:CS}

We first in Section \ref{s:basics} describe the basics of cyclic algebras, following Pierce \cite{Pie}.   We then in Section \ref{s:construction} extend the construction of \cite{CMSZ1} and \cite{CS} to general cyclic extensions, using invariant language.  For brevity, we will refer to the construction in \cite{CMSZ1} and \cite{CS} as the Cartwright--Steger construction.   Finally in Section \ref{s:applications} we restrict to the case of finite fields and recall or prove facts that will be useful for our constructions of lattices in Sections \ref{s:d=3} and \ref{s:d>3} below.

\subsection{Basic definitions and properties}\label{s:basics}

Let $\mE \supseteq \mK$ be a cyclic Galois extension of degree $d$, $\sigma \in \Gal (\mE / \mK)$ a generator and $a\in \mK^\times$ an element.   
The cyclic algebra $(\mE,\sigma,a)$ is generated as a ring by $\mE$ and an extra element $t$, with $\mE$ a subring so that the ring operations of $\mE$ are retained  in $(\mE,\sigma,a)$. 
The relations involving $t$ are
$$
t^d=a, \ \ 
tb = \sigma (b)t \ \mbox{ for all } b\in\mE. 
$$ 
The following are well-known properties of the cyclic algebras:
\begin{mylist}
\item[(1)] $(\mE,\sigma,a)$ is a central simple algebra over $\mK$ of dimension $d^2$;
\item[(2)] $\mE$ is a maximal subfield of $(\mE,\sigma,a)$; and
\item[(3)] the elements $1, t, t^2, \ldots, t^{d-1}$ form a basis of $(\mE,\sigma,a)$ over $\mE$.
\end{mylist}
In particular, each cyclic algebra defines an element $[(\mE,\sigma,a)]$ in the relative Brauer group $\Br(\mE /\mK)$.
Recall the definitions of the trace and the norm $T,N: \mE \rightarrow \mK$:
$$
T(a)=\sum_{k=0}^{d-1} \sigma^k (a), \ \ 
N(a)=\prod_{k=0}^{d-1} \sigma^k (a).
$$
The norm image $N(\mE^\times)$ is a subgroup of $\mK^\times$. 
We also need the following properties~\cite{Pie}:
\begin{mylist}
\item[(4)] $(\mE,\sigma,a)\cong M_d(\mK)$ if and only if $a\in N(\mE^\times)$; and
\item[(5)] if $a\in\mK^\times$ and the order of  $aN(\mE^\times)\in \mK^\times /N(\mE^\times)$ is $d$ then $(\mE,\sigma,a)$ is a division algebra.
\end{mylist}

The cyclic extension $\mE \supseteq \mK$ gives rise to two further cyclic Galois extensions:
the fields of rational functions $\mE(Y) \supseteq \mK(Y)$ and
the fields of Laurent series
$\mE ((Y)) \supseteq \mK ((Y))$.
One can think of them as Galois extensions with the same Galois group, so that 
$\sigma$ acts on the coefficients while $\sigma (Y)=Y$.

\subsection{The construction}\label{s:construction}

The first cyclic algebra of interest to us is 
$$\cA:=(\mE(Y),\sigma,1+Y).
$$
It is a division algebra,
by property (5) \cite[p.84]{Jaco}: the equation
$$
N\left(\frac{a_0+\cdots +a_m Y^m}{b_0+\cdots +b_k Y^k}\right) = (1+Y)^n
$$
with $a_m\neq 0\neq b_k$ gets rewritten as
$$
N(a_m)Y^{md} + \cO (Y^{md-1}) = ( N(b_k) Y^{kd} + \cO (Y^{kd-1}))(Y^{n} + \cO (Y^{n-1})).
$$
Comparing the highest terms, $md=kd+n$. Hence $n$ must be divisible by $d$, to be a norm of some element.
Since $N(1+Y)=(1+Y)^d$, the order of $(1+Y)N(\mE^\times)$ is exactly $d$. By (5), $\cA$ is a division algebra.

The second cyclic algebra of interest is 
$$
\widetilde{\cA}:=(\mE((Y)),\sigma,1+Y)\cong \mK((Y))\otimes_{\mK(Y)}\cA.
$$
It is isomorphic to the matrix algebra $M_d (\mK((Y)))$ by (4).
To observe this, let us note that the trace $T: \mE((Y)) \rightarrow \mK((Y))$ is surjective.
Indeed, pick any $x\in \mE((Y))$ with nonzero trace $T(x)=\beta \in \mK((Y))$, then for every $\alpha\in \mK((Y))$
we have
$T(\alpha  \beta^{-1} x )=\alpha$.
This allows to solve the equation
$$
N( 1 + x_1Y + x_2 Y^2 + \cdots ) = 1+Y
$$
recursively: $x_1$ is a solution of $T(x_1)=1$, and each consecutive term $x_n$ will be a solution of
$T(x_n) = f_n (x_1, \ldots, x_{n-1})$ for a certain function $f_n$ of all the previously found terms.

We would like to write an explicit isomorphism $\Psi$
from $\widetilde{\cA}$ to a matrix algebra. 
Observe that in $\widetilde{\cA}$ for any $a,b\in \mE ((Y))$
$$
(at)b=\sigma(b)at \ \
\mbox{ and } \ \
(at)^d=a\sigma(a)t^2(at)^{d-2}=\cdots =N(a)t^d=N(a)(1+Y)
.
$$
Hence, if $X\in\mE ((Y))$ is a solution of $N(X)=1+Y$ then
$$
\sum_j a_j t^j \mapsto \sum_j a_j X^j \hat{t}^j
$$
is an isomorphism from $\widetilde{\cA}$  to  
$(\mE((Y)),\sigma,1)$.
The latter is known as the 
{\em skew group algebra}
and admits an explicit isomorphism
%
to the matrix algebra $\End_{\mK((Y))}(\mE((Y))$ given by
$a\hat{t}^j : b \mapsto a \sigma^j (b)$. Composing these isomorphisms, we arrive at an explicit isomorphism
$$ \Psi : \widetilde{\cA} \rightarrow \End_{\mK((Y))}(\mE((Y))) $$
given by
\begin{equation}\label{e:Psi}
\Psi \left(\sum_j a_j t^j\right): b \mapsto \sum_j a_j X^j \sigma^j(b), \ \ b\in\mE((Y))  .
\end{equation}

We will abuse notation by denoting various restrictions of $\Psi$, for instance to $\cA$, by the same letter. 
On the level of multiplicative groups
we have an injective homomorphism
$$
\Psi : \cA^\times \rightarrow \GL_{\mK((Y))}(\mE((Y))).
$$
By the Skolem--Noether Theorem, every $\mK(Y)$--linear
automorphism of $\cA$ is inner, so we have another injective group homomorphism
$$
\oP : \Aut(\cA)\cong \cA^\times /Z(\cA^\times) \rightarrow \PGL_{\mK((Y))}(\mE((Y))).
$$

Now we are ready to introduce the {\em Cartwright--Steger groups} \cite{CMSZ1,CS}.
Let $\cA_0$ be the $\mE [Y^{-1}]$--span of the elements $t^m$, $m<d$ in $\cA$. Notice that it is not a subring: $t^d=1+Y\not\in \cA_0$.
The ``big" Cartwright--Steger group $\widetilde{\G}$ is defined as 
$$
\widetilde{\Gamma} := \{ \gamma \in  \Aut(\cA)\;\mid\; \gamma (\cA_0)\subseteq \cA_0 \}.
$$
Why is $\widetilde{\Gamma}$ a subgroup? 
To show this we choose a $\mK(Y)$--basis $\cB$  of $\cA$ consisting of the elements
$at^m$, $m<d$, $a\in \mE$. The basis $\cB$ is also a $\mK((Y))$--basis of $\widetilde{\cA}$.
Writing automorphisms in this basis gives an injective homomorphism
$$
\Phi : \Aut(\cA) \rightarrow \GL_{\mK(Y)}(\cA) \rightarrow \GL_{\mK((Y))}(\widetilde{\cA}) \cong \GL_{d^2} (\mK((Y))).
$$
Moreover, each $\Phi (\gamma)$ is an automorphism of $\widetilde{\cA}$.
By the Skolem--Noether Theorem,
$\Phi (\gamma) (x) = y_\gamma x{y_\gamma}^{-1}$ for a certain $y_\gamma\in \widetilde{\cA} \cong M_d (\mK((Y)))$.
It follows that $\det(\Phi (\gamma)) = \det (y_\gamma)^d \det (y_\gamma)^{-d} = 1$ \cite[p.129]{CS}.
Thus, we can restrict the image of $\Phi$ to the special linear group:
$$
\Phi : \Aut(\cA) \rightarrow \SL_{d^2} (\mK((Y))).
$$
Clearly, $\gamma\in\widetilde{\Gamma}$ if and only if the coefficients of $\Phi (\gamma)$ lie in $\mK [Y^{-1}]$.
Thus,
$$
\widetilde{\Gamma} = \Phi^{-1} (\SL_{d^2} (\mK[Y^{-1}]))
$$
is a subgroup.  
Since $\gamma\in\widetilde{\Gamma}$ is $\mK(Y)$--linear, we have 
$\gamma (Y^{-1}\cA_0)\subseteq Y^{-1}\cA_0$ for any $\gamma \in \widetilde{\Gamma}$.
Thus $\gamma$ defines a linear map $\Theta (\gamma) \in \mbox{End}_\mK (\cA_+)$ where
$\cA_+ = \cA_0 / Y^{-1}\cA_0$. The map $\Theta$ is a semigroup homomorphism from a group, 
so its image consists of invertible elements:
$$
\Theta : \widetilde{\Gamma}
\rightarrow \GL_\mK (\cA_+) \cong \GL_{d^2} (\mK).
$$
In essence, $\Theta$ is the $Y$--degree zero term of $\Phi$:
the basis $\cB$ defined above gives an $\mK$--basis of $\cA_+$.
The basis $\cB$ has a partial order coming from the degree of $t$ in $[at^j]=at^j+ Y^{-1}\cA_0$.
Let $T$ be the group of ``unitriangular'' transformations in this basis, that is,
$$
T = \{ \pi \in \GL_\mK (\cA_+) \;\mid\; \forall a\in \mE, j<d \ \ 
\pi ([at^j]) = [at^j] +\sum_{i=0}^{j-1} [a_it^i], \  a_i\in \mE \} .
$$
Finally, the ``small" Cartwright--Steger group is
$$
\Gamma := \Theta^{-1} (T) \leq \widetilde{\G}.
$$
(Since not all of $T$ may be in the image of $\Theta$, we should perhaps write that $\G = \Theta^{-1}(T) \cap \mbox{Im}(\Theta)$.)

\begin{lemma}\label{l:gamma}
If $\gamma\in\Gamma$ then 
$$
 \gamma (t)=t+\cO (Y^{-1})     
\ \ 
\mbox{ and }
\ \ 
\gamma(t^{d-1}) = t^{d-1}+\cO (Y^{-1})
$$
where
$\cO ({Y^{-1}})$ denotes a polynomial
in negative degrees of $Y$ with coefficients in $\cA_0$. 
\end{lemma}

\begin{proof}
By definition of $\Gamma$, 
$$
\gamma (t) = t+ a +\cO (Y^{-1})  \ \ \mbox{and} \ \ 
\gamma(t^{d-1}) = t^{d-1}+ b_{d-2}t^{d-2} + \cdots +b_1 t + b_0 +\cO (Y^{-1})
$$
for some $a,b_i\in \mE^\times$.
Let us analyse the key equation  
$$
1+Y = \gamma (1+Y) = \gamma (t^d) = \gamma(t)\gamma(t^{d-1}).
$$
Since
$t^d =1+Y$ we get the equation
\begin{dmath*}
(a+\sigma(b_{d-2}))t^{d-1} + (ab_{d-2}+\sigma(b_{d-3}))t^{d-2} +  \cdots  + (ab_1+\sigma(b_{0}))t^{1} + ab_0 + \cO (Y^{-1}) = 0.
\end{dmath*}
If $a=0$ then we immediately conclude that all $\sigma (b_i)=0$. Hence all $b_i=0 $ and we are done.
If $a\neq 0$ then we conclude that all $b_0=0$. Then $b_1=0$. Recursively, all $b_i=0$ and we are done.
\end{proof}

To contemplate the difference between $\Gamma$ and $\widetilde{\Gamma}$, let us introduce another group $H$:
as a set $H$ consists of $\gamma\in \mbox{Aut}(\cA)$ that are  conjugations by $at^j$, where $a\in\mE$ and $j<d$.
 
\begin{prop}
\label{p:Dm1}
\begin{mylist}
\item[(1)] $H$ is a subgroup of $\widetilde{\Gamma}$.
\item[(2)] $H\cap\Gamma=\{ 1\}$.
\item[(3)] $H\Gamma$ is a subgroup of $\widetilde\G$ and $\Gamma$ is normal in $H\Gamma$.
\end{mylist}
\end{prop}

\noindent As recalled in Section \ref{s:applications} below, in the case of finite fields $H\Gamma=\widetilde{\Gamma}$, which may or may not hold over arbitrary fields.  This is an interesting question.

\begin{proof}
Let us calculate in $\cA$, writing $x\sim y$ when $x$ and $y$ give the same conjugation in $\mbox{Aut}(\cA)$.
Since $1+Y\sim 1$,
$$
(at^j)^{-1} = t^{d-j}a^{-1}(1+Y)^{-1} \sim \sigma^{d-j}(a^{-1})t^{d-j}
\ \ \mbox{ and }
$$
$$
(at^j)(bt^i)=a\sigma^j(b)t^{i+j}\sim a\sigma^j(b)t^{i+j-d}
,$$
showing that $H$ is a subgroup of $\mbox{Aut}(\cA)$.
If $\gamma\in H$ is a conjugation by $at^j$, where $a\in\mE$ and $j<d$, then 
\begin{dmath*}
\gamma (bt^i) = at^jbt^it^{d-j}a^{-1}(1+Y)^{-1} = a\sigma^j(b)t^{d+i}a^{-1}(1+Y)^{-1} = a\sigma^i(a^{-1})\sigma^j(b)t^{i}.
\end{dmath*}
Thus $H$ is a subgroup of $\widetilde{\Gamma}$. 
Moreover, 
$\gamma \in \Gamma$ if and only if $b = a\sigma^i(a^{-1})\sigma^j(b)$ for all $b$ and $i$
if and only if
$a\in \mK$ and $j=0$ if and only if $\gamma=1$. 
This proves (2). 

Finally, it suffices to check that $\gamma \Gamma \gamma^{-1} \subseteq \Gamma$
where  $\gamma$ is a conjugation by $x$, and $x$ is either $t$ or $a\in\mE^\times$.
If $\beta\in \Gamma$, then 
$$
\gamma \beta \gamma^{-1}(y) = x \beta (x^{-1}) \beta (y) (x\beta (x^{-1}))^{-1}.
$$
Note that elements of $\Gamma$ are characterised by the fact that
$$
\beta (bt^i) = bt^i + \cO(t^{i-1}) + \cO ({Y^{-1}})
$$
for all $b\in\mE$ and $i\in\{0,1,\ldots d-1\}$, where $\cO(t^{i-1})$ denotes a polynomial
in $1,t\ldots t^{i-1}$ with coefficients in $\mE$ and $\cO ({Y^{-1}})$ denotes a polynomial
in negative degrees of $Y$ with coefficients in $\cA_0$. 

If $x=a$ then
$$
\beta (a) = a + \cO ({Y^{-1}}), \ \ 
\beta (a^{-1}) = a^{-1} + \cO ({Y^{-1}})
$$
by the definition of $\Gamma$ and
$$
x \beta (x^{-1})=1 + \cO ({Y^{-1}}), \ \
(x \beta (x^{-1}))^{-1}= \beta (x)x^{-1} = 1 + \cO ({Y^{-1}}).
$$
Finally, 
\begin{dmath*}
\gamma \beta \gamma^{-1}(y) = 
\Big(1 + \cO ({Y^{-1}})\Big) \Big(bt^i + \cO(t^{i-1}) + \cO ({Y^{-1}})\Big)\Big(1 + \cO ({Y^{-1}})\Big) = bt^i + \cO(t^{i-1}) + \cO ({Y^{-1}})
\end{dmath*}
because there would not be enough powers of $t$  to cancel all of the $Y^{-j}$ using  $t^d = 1+Y$ and produce at least an $i$-th power of $t$.

Similarly, if $x=t$ then
$$
\beta (t) = t + \cO ({Y^{-1}}), \ \ 
\beta (t^{-1}) = (1+Y)^{-1} (t^{d-1} + \cO ({Y^{-1}}))
$$
by Lemma \ref{l:gamma}. Since $(1+Y)^{-1} = Y^{-1}-Y^{-2}+Y^{-3}-\cdots$, 
$$
x \beta (x^{-1})=1 + \cO ({Y^{-1}}), \ \
(x \beta (x^{-1}))^{-1}= \beta (x)x^{-1} = 1 + \cO ({Y^{-1}}).
$$
Finally, 
$$
\gamma \beta \gamma^{-1}(y) =
bt^i + \cO(t^{i-1}) + \cO ({Y^{-1}})
$$
as in the case of $x=a$.
\end{proof}


It would be useful for us to know how the image $\oP (\widetilde{\Gamma})$ intersects with 
$\PSL_{\mK((Y))}(\mE((Y)))$. 
We can understand this for the image of $H$.
By $(\mK^\times )^k$ we denote the subgroup of the multiplicative group $\mK^\times$
consisting of $k$-th powers.
Let 
$\gamma: \cA^\times \rightarrow {\Aut}(\cA)$ be the homomorphism assigning the conjugation by $x$ to each $x\in\cA^\times$.

\begin{prop}
\label{p:Dm2} 
Let $p$ be the characteristic of $\mK$. Denote by $\Ord_p (m)$ the largest power of $p$ that divides an integer $m$
(or $1$ if $p=0$).
Then
$$
\oP (H) \cap \PSL_{\mK((Y))}(\mE((Y))) = 
$$
$$
\{ \Psi(\gamma (at^k)) \;\mid\; a\in\mE^\times, N(a)\in (\mK^\times)^d, 
\Ord_p (k) \geq \Ord_p (d)  \}.
$$
\end{prop}

\begin{proof}
The element 
$\oP(\gamma (at^k))$ is in $\PSL_{\mK((Y))}(\mE((Y)))$ if and only if one can multiply $\Psi(at^k)$ by a scalar matrix $zI_d$, $z\in\mK ((Y))$, so that the determinant
of the product is 1. Now the product 
$$
z \Psi (a t^k): b \mapsto za X^k \sigma^k(b), \ \
\forall b\in \mE ((Y))
$$
is a composition of four linear maps
$$
(b\mapsto zb)
\circ
(b\mapsto ab)
\circ
(b\mapsto X^kb)
\circ
\sigma^k
$$
so its determinant is the product of four determinants:
$$
\det (z \Psi (a t^k))
=
z^d \cdot
N(a) \cdot
(1+Y)^k \cdot
(-1)^{(d-1)k} .
$$
Here we use the fact that the determinant of the multiplication
$(b\mapsto ab)$ is the norm $N(a)$. In particular, we see three norms,
including
$N(z)=z^d$ and $N(X^k)=(1+Y)^k$.
From Galois theory, we know that the action of $\sigma$ on $\mE ((Y))$ is conjugate to
the permutation matrix of a cycle of length $d$ that gives the last determinant.

Thus, we just need a $d$-th root of $(-1)^kN(a) (1+Y)^k$ in $\mK((Y))$.
The free term of such a root is a $d$-th root of $N((-1)^ka)$. 
Therefore it is necessary and sufficient to have 
$d$-th roots of both $N((-1)^ka)$ and  $(1+Y)^k$.
The existence of the former is equivalent to 
$N((-1)^ka)\in (\mK^\times)^d$, while the existence of the latter
is equivalent to 
$\Ord_p (k) \geq \Ord_p (d)$.

The last statement needs an explanation.
Write $d={\mathrm{Ord}}_p (d)d^\prime$.
Extracting a $d^\prime$-th root of $(1+Y)^k$ can be done because $d^\prime$ is invertible in $\mK$:
the equation
$$
( 1 + x_1Y + x_2 Y^2 + \ldots )^{d^\prime} = (1+Y)^k
$$
can be solved recursively: $x_1$ is a solution of $d^\prime x_1=k$, and each consecutive term $x_n$ will be a solution of
$d^\prime x_n = f_n (x_1, \ldots, x_{n-1})$ for a certain function $f_n$ of all the previously found terms.
It remains to contemplate extracting of the $p$-th root in characteristic $p$: since
$$
( 1 + x_1Y + x_2 Y^2 + \cdots )^{p} = 1+ x_1^p Y^p + x_2^p Y^{2p} + \cdots
$$
this can be done if and only if $(1+Y)^k$ is already a $p$-th power, that is, if and only if $p$ divides $k$.

Finally, since  $\Psi(\gamma ((-1)^kat^k)) = \Psi(\gamma (at^k))$
we can replace $(-1)^ka$ with $a$.
\end{proof}

\subsection{Application to the case of finite fields, and summary of useful results}\label{s:applications}

While the algebraic properties of the construction in Section \ref{s:construction} above are upheld in any cyclic extension, we would like to move to its topological and metric properties. For this, from now on we assume that the extension $\mE \supseteq \mK$ is a finite field extension $\Fqd \supseteq \Fq$ with $q=p^a$, $p$ a prime.  

\begin{prop}
\label{p:Dm3} 
Let $\mE = \Fqd$ and $\mK= \Fq$. Then
$$
\left| \oP (H) :
\left(
\oP (H) \cap {\mathrm{PSL}}_{\mK((Y))}(\mE((Y)))
\right) \right| =
\delta\cdot \Ord_p (d)
$$
where 
$\delta$ is the greatest common divisor 
of $d$ and $(q-1)$ (note that $\delta$ is a divisor of $(q^d-1)/(q-1)$).
\end{prop}
\begin{proof}
Clearly $at^k \sim bt^m$ (with $k,m<d$) if and only if $ab^{-1}\in \mK$ and $k=m$.
Thus, we can compute the contributions to the index from $a$ and from $t$ separately.
The powers of $t$ of degrees $\Ord_p (d)$, $2\Ord_p (d)$, \ldots, $d - \Ord_p (d)$
are exactly those that produce elements of the subgroup.
So, $\Ord_p (d)$ 
is the contribution from $t$.
The contribution from $a$ is the index
$$
\left| \mE^\times  : \mK^\times N^{-1} ((\mK^\times)^d ) 
\right| 
=
\left| \mE^\times  : N^{-1} ((\mK^\times)^d ) 
\right| 
=
\left| \mK^\times  : (\mK^\times)^d  
\right|
=
n.
$$
The first equality holds because $\mK^\times \subseteq N^{-1} ((\mK^\times)^d )$.
Indeed,  $N(a)=a^d\in (\mK^\times)^d$ for all $a\in \mK^\times$.
The second equality holds
since $N$ is surjective and
$(\mK^\times)^d$ has index $n$ in $\mK^\times$.
\end{proof}

Using the explicit expression for $\Psi$ at \eqref{e:Psi} above, one can construct an explicit image of $H$ in the locally compact, totally disconnected group $G = \PGL_d(\Fqt)$ under $\overline{\Psi}$. 
Interestingly enough, if $(p,d)=1$, one can see that $\overline{\Psi}(H)$ can be realised as a subgroup of 
$\PGL_d(q)$ naturally embedded in 
$\PGL_d(\Fq[[t]])$.
However, if $p\mid d$, this is not possible and 
$\overline{\Psi}(H)\cap \PGL_d(q)$ is a subgroup of index $\Ord_p(d)$ in $\overline{\Psi}(H)$.
This difference comes from the fact that in the former case $X$ (a solution of $N(X) = 1 + Y$) can be realised over $\Fq$, while in the latter case this is not possible.

So far we have been working in $\Aut(\widetilde{\mathcal{A}})$. However, 
it will now be convenient to switch our discussion explicitly into $G=\PGL_d(\Fqt)$.  
To avoid excessive notations, we identify $\widetilde{\G}$ with its image $\overline{\Psi}(\widetilde{\G})$ in $G$.  From now on we call this image  $\widetilde{\G}$. Likewise, we call  $\widetilde{\G}_v$, now in $G$, again by $H$ (instead of using $\overline{\Psi}(H)$).

We now recall the facts about $\widetilde\G$ that will be useful for us.  
Most of them can be derived from Section \ref{s:construction} but, as they already appear in \cite{CS},  we just restate them.  We have:

\begin{enumerate}
\item $\widetilde\G$ is a cocompact lattice of $\PGL_d(\Fqt)$;
\item $\G$ acts simply transitively on the set of vertices of the building $\Delta$ associated to $\PGL_d(\Fqt)$;
\item $H=\widetilde{\G}_{v}$ for a vertex $v$ of $\Delta$;
\item $|H|=\frac{q^d-1}{q-1}d$; and 
\item $\widetilde{\G}=H\G$.
\end{enumerate}

We will now discuss the structure of $H$ and some of its properties.

\begin{lemma}
\label{l:H}
Let 
$H=\widetilde{\G}_v$ for a vertex $v$ of $\Delta$ the building associated to $G = \PGL_d(\Fqt)$.  Then the following conditions hold:

\begin{enumerate} 
\item  $H$ is  a subgroup  of $G_v\cong\PGL_d(\Fq[[t]])$;
\item $H$ contains a normal cyclic subgroup $S$ of order $\frac{q^d-1}{q-1}$ where 
$S$ is a Singer cycle of $\PGL_d(q)$;
\item $H\cong N_{\PGL_d(q)}(S)$; and 
\item  if we identify $\PSL_d(\Fqt)$ with a subgroup of $G$, then $$|H\cap\PSL_d(\Fqt)|=\frac{d}{\Ord_p(d)}\cdot\frac{q^d-1}{(q-1)(d,q-1)}.$$

\end{enumerate}
\end{lemma}

\begin{proof} Part (1) follows immediately from the fact that $H=\widetilde{\G}_v$, hence $H\leq G_v$, and the fact that $G_v\cong \PGL_d(\Fq[[t]])$, as discussed in Section \ref{s:buildings}.

For (2), using the notation of Proposition~\ref{p:Dm1}, let $S$ be the image of $mE^\times$ in $\PGL_d(q)$.
Obviously,  $S$ is a cyclic subgroup of $H$ of order $\frac{q^d-1}{q-1}$.
Now from the proof of (1) of Proposition \ref{p:Dm1}, it follows that $S$ indeed is normal in $H$.
Moreover,  as $S$ is an abelian subgroup of $\PGL_d(q)$ of order $\frac{q^d-1}{q-1}$, Proposition 2.2 of \cite{CdR} implies that $S$ is a Singer cycle of $\PGL_d(q)$.

To prove (3), we have $H\leq G_v\cong \PGL_d(\Fq[[t]])\cong U_v\rtimes \PGL_d(q)$ where $U_v$ is a pro--$p$ group.
If $(p,d)=1$, then $(|H|,p)=1$ and so $H\cap U_v=1$. Suppose that $p\mid d$. Assume that $H\cap U_v\neq 1$. Then there exists $1\neq h\in H\cap U_v$, an element  of order
$p$. It follows that $[h,S]\leq U_v\cap S=1$ since on the one hand $h\in U_v\triangleleft G_v$ and $S\leq G_v$, while on the other, $h$ normalises $S$ and $(p,|S|)=1$.  Thus $h$ centralises $S$. Using calculations from the proof of Proposition \ref{p:Dm1}(1) we observe that  $S$ is self-centralising in $H$. We have reached a contradiction that proves that $H\cap U_v=1$.
It follows immediately that $H\cong \overline{H}\leq \overline{G}_v:=G_v/U_v\cong \PGL_d(q)$. 

Now $\overline{H}$ contains a normal subgroup $\overline{S}\cong S$ which is a Singer cycle of $\overline{G}_v$, by 
Proposition 2.2 of \cite{CdR}. Moreover, $|\overline{H}|=|N_{\PGL_d(q)}(\overline{S})|$. Therefore (3)  holds.

Finally using (1), (2) and (3) together with Proposition \ref{p:Dm3}, we conclude that (4) holds.
\end{proof}



\section{Lattices in case $d = 3$}\label{s:d=3}

In this section we prove Theorems \ref{t:PGL} and \ref{t:PSL} in the case $d = 3$.  We construct and establish the properties of  lattices $\G'_0 \leq \G_0$ in $\PGL_3(\Fqt)$ in Section \ref{s:PGL_3}, and investigate the intersections $\Lambda'_0:=\G'_0 \cap \PSL_3(\Fqt)$ and $\Lambda_0:=\G_0 \cap \PSL_3(\Fqt)$ in Section \ref{s:PSL_3}.  

\subsection{Lattices in $\PGL_3(\Fqt)$}\label{s:PGL_3}

Recall the construction of the cocompact lattice $\tG\leq \PGL_3(\Fqt)$ described in Section~\ref{s:CS} above. As noted in Section \ref{s:applications}(5) above, the lattice $\tG$ is a product of a vertex stabiliser $H$ of order $3(q^2 + q + 1)$, and a vertex-regular lattice~$\G$.  By Lemma \ref{l:H} above, $H$ contains a Singer cycle $S$ of $\PGL_3(q)$.  Denote by $\tG'$ the subgroup of $\tG$ which is the product of $S$ and $\G$.  Then by construction, $S$ is a vertex stabiliser in $\tG'$.  (Since $\G \leq \tG' \leq \tG$, the group $\tG'$ is also a cocompact lattice in $\PGL_3(\Fqt)$.)

Let $v_0$, $v_1$ and $v_2$ be the vertices of the standard chamber of $\Delta$, as in Section \ref{s:buildings} above.  For $i = 0,1,2$ let $N_i$ be the stabiliser of $v_i$ in $\tG$, and let $S_i$ be the stabiliser of $v_i$ in $\tG'$.  Since $\tG$ and $\tG'$ act transitively on the vertices of $\Delta$, we have that each $N_i \cong H$ and each $S_i \cong S$.  We now define 
\[ \G'_0 := \la S_0, S_1, S_2 \ra \]
to be the subgroup of $\tG'$ generated by $S_0$, $S_1$ and $S_2$, and 
\[ \G_0 := \la N_0, N_1, N_2 \ra \]
to be the subgroup of $\tG$ generated by $N_0$, $N_1$ and $N_2$.  Clearly $\G'_0 \leq \G_0$.  

We claim that $\G'_0$ and $\G_0$ are cocompact lattices in $G=\PGL_3(\Fqt)$.  Recall from the introduction that $\G < G$ is a cocompact lattice in $G$ if it is a discrete subgroup of $G$ which acts cocompactly on $\Delta$.  Hence it suffices to show that $\G_0$ is a discrete subgroup of $\PGL_3(\Fqt)$ and that $\G'_0$ acts cocompactly on $\Delta$.  The following lemma is immediate, since by construction $\G_0$ is a subgroup of the discrete group $\tG \leq \PGL_3(\Fqt)$.

\begin{lemma} $\G_0$ is a discrete subgroup of $\PGL_3(\Fqt)$.
\end{lemma}

To show that $\G'_0$ acts cocompactly on $\Delta$, we first consider the action of the groups $S_i$ which generate $\G'_0$.

\begin{lemma}\label{l:neighbours}  For $i = 0,1,2$ and $j = i-1, i+1 \pmod 3$, the group $S_i$ acts simply transitively on the vertices neighbouring $v_i$ of type $j$.
\end{lemma}

\begin{proof}
From the discussion of Singer cycles in Section \ref{s:singer} and types in Section \ref{s:buildings}, the group $S_0$ acts simply  transitively on the vertices neighbouring $v_0$ of type $j$, for $j = -1,1 \pmod 3$.  Now $\tG'$ consists of type-rotating automorphisms, since the Cartwright--Steger lattice $\tG$, which contains $\tG'$, consists of type-rotating automorphisms.  By construction and the definition of type-rotating, for $i = 1,2$ the group $S_i$ is the image of $S_0$ under conjugation by an element of $\tG'$ which adds $i \pmod 3$ to each type.  Thus for $i = 1,2$, the group $S_i$ acts simply transitively on the vertices neighbouring $v_i$ of type $j = i-1, i+1 \pmod 3$. \end{proof}

\begin{prop}\label{p:vertex transitive} For $i = 0,1,2$, the group $\G'_0$ acts transitively on the vertices of type $i$ in $\Delta$.
\end{prop}

\begin{proof}  We will show that $\G'_0$ acts transitively on the vertices of type $0$ in $\Delta$.  The same argument will apply for types $1$ and $2$.

It suffices to show that for each vertex $w_0$ of type $0$, there is an element of $\G'_0$ which takes $w_0$ to $v_0$.  We prove this by induction on the distance from $w_0$ to $v_0$ in the natural graph metric $\delta$ on the edges of $\Delta$.  Note that $\delta(w_0,v_0)$ will always be an even integer since no two vertices of type $0$ are adjacent.

If $\delta(w_0,v_0) = 2$ we consider two cases.  The first is when $w_0$ is adjacent to either $v_1$ or $v_2$.   By Lemma \ref{l:neighbours} above, $S_1$ and $S_2$ act transitively on the type $0$ neighbours of $v_1$ and $v_2$ respectively, and so the claim follows in this case.  Otherwise, $w_0$ is adjacent to some vertex $s_0 v_1$ or $s_0' v_2$ where $s_0, s_0' \in S_0$, since $S_0$ acts transitively on the vertices of types $1$ and $2$ which neighbour $v_0$.  Then $s_0^{-1} w_0$ is adjacent to $v_1$ or $(s_0')^{-1}w_0$ is adjacent to $v_2$, and we apply the argument from the first case.

Now suppose that $\delta(w_0,v_0) = 2k$.  Then there is a vertex $w_0'$ of $\Delta$ of type $0$ such that $\delta(w_0,w_0') = 2(k-1)$ and $\delta(w_0',v_0) = 2$.  By the base case of the induction there is an element $\gamma \in \hat\G_0$ such that $\gamma w_0' = v_0$.  But then $\delta(\gamma w_0, v_0) = \delta(\gamma w_0, \gamma w_0') = \delta( w_0,w_0') = 2(k-1)$ so by inductive assumption there is a $\gamma' \in \hat\G_0$ such that $\gamma' \gamma w_0 = v_0$, as required.
\end{proof}

\begin{corollary}\label{c:cocompact} $\G'_0$ acts cocompactly on $\Delta$.
\end{corollary}

\begin{proof} By Proposition \ref{p:vertex transitive} above, $\G'_0$ has finitely many (at most $3$) orbits of vertices on $\Delta$.  Since $\Delta$ is locally finite, this implies that $\G'_0$ acts cocompactly.
\end{proof}

We have established the claim that $\G'_0$ and $\G_0$ are cocompact lattices in $\PGL_3(\Fqt)$.  To finish the proof of Theorem \ref{t:PGL} in the case $d = 3$, we further describe the actions of $\G'_0$ and $\G_0$ on $\Delta$.

\begin{corollary}\label{c:action}  The action of $\G'_0$ and of $\G_0$ is type-preserving and transitive on each type of vertex in $\Delta$.  For $i = 0,1,2$, the stabiliser of $v_i$ in $\G_0'$ is the group $S_i$, and the stabiliser of $v_i$ in $\G_0$ is the group $N_i$.   
\end{corollary}

\begin{proof}   Each $N_i$ is a subgroup of the type-rotating group $\tG$ and stabilises a vertex of type $i$, hence each $N_i$ fixes all types.  It follows that $\G_0$ and thus $\G'_0$ is type-preserving.  By Proposition \ref{p:vertex transitive}, the action of $\G'_0$ and thus of $\G_0$ is transitive on each type of vertex of $\Delta$.  For $i = 0,1,2$,  the stabiliser of $v_i$ in $\G'_0$ is $S_i$ since by construction $$S_i \leq \Stab_{\G'_0}(v_i) \leq \Stab_{\tG'}(v_i) = S_i.$$  Similarly, the stabiliser of $v_i$ in $\G_0$ is $N_i$.
\end{proof}

\subsection{Lattices in $\PSL_3(\Fqt)$}\label{s:PSL_3}

We will first prove that $\Lambda_0:=\G_0 \cap \PSL_3(\Fqt)$ is a cocompact lattice in $\PSL_3(\Fqt)$.  The proof that $\Lambda'_0 :=\G'_0 \cap \PSL_3(\Fqt)$ is a cocompact lattice in $\PSL_3(\Fqt)$ is similar.  

Since $\G_0$ is discrete, it is immediate that $\Lambda_0$ is a discrete subgroup of $\PSL_3(\Fqt)$.  Now $\G_0$ acts cocompactly on $\Delta$, so to show that $\Lambda_0$ act cocompactly on $\Delta$ it suffices to show that $\Lambda_0$ is of finite index in $\G_0$. 

Consider the determinant homomorphism $\det:\GL_3(\Fqt) \to \Fqt^\times$, with kernel $\SL_3(\Fqt)$.  This homomorphism induces a well-defined homomorphism
\[\overline\det: \PGL_3(\Fqt) \to \Fqt^\times/(\Fqt^\times)^3\]
where $(\Fqt^\times)^3$ is the subgroup of $\Fqt^\times$ consisting of cubes of invertible elements of $\Fqt$.  The kernel of $\overline\det$ is $\PSL_3(\Fqt)$.  

The group $\G_0$ is finitely generated by torsion elements, since each $N_i$ is finite.  Hence the restriction of $\overline\det$ to $\G_0$ has finite image.  But the kernel of this restriction is $\G_0 \cap \PSL_3(\Fqt) = \Lambda_0$.  Thus $\Lambda_0$ has finite index in $\G_0$, as required.  We conclude that $\Lambda_0$ is a cocompact lattice in $\PSL_3(\Fqt)$.

Our further discussion is divided into cases depending upon the value of $q$.  We will establish the remaining claims of Theorem \ref{t:PSL} and specify the relationship between our lattices and the Cartwright--Steger lattice $\tG$ in Sections \ref{s:3 divides q+1} and \ref{s:3 divides q}, then in Section \ref{s:3 divides q-1} explain why, if $(3,q-1) \neq 1$, we are not able to describe any more precisely the actions of $\Lambda_0$ and $\Lambda'_0$.

\subsubsection{Case $3 \mid (q+1)$}\label{s:3 divides q+1}

Note that in this case $(d,q-1) = 1$ and $p \neq 3$, so in particular $p$ does not divide $d = 3$.  

By Lemma~\ref{l:H}(4) above, in this case we have \[ |H \cap \PSL_3(\Fqt)| = 3(q^2 + q + 1) = |H|\] and so the entire group $H$ is contained in $\PSL_3(\Fqt)$.  The groups $N_i$ which generate $\G_0$ are by construction conjugates of $H$ in $\tG \leq \PGL_3(\Fqt)$, hence for $i = 0,1,2$ the group $N_i$ is also contained in $\PSL_3(\Fqt)$.  Therefore $\G_0 = \la N_0, N_1, N_2 \ra$ is contained in $\PSL_3(\Fqt)$, that is, $\Lambda_0 = \G_0$.  Since $\G_0' \leq \G_0$, we also have that $\G'_0$ is contained in $\PSL_3(\Fqt)$, that is, $\Lambda'_0 = \G'_0$.  By the same arguments as in Section \ref{s:PGL_3} above, it follows that $\G'_0$ and $\G_0$ are cocompact lattices in $\PSL_3(\Fqt)$ with action as described in Corollary \ref{c:action} above.  

We can now specify the relationship between our lattice $\G_0$ and the Cartwright--Steger lattice $\tG$, in this case.

\begin{lemma}\label{l:3 divides q+1} If $3$ divides $(q+1)$, then $\G_0 = \tG \cap \PSL_3(\Fqt)$.
\end{lemma}

\begin{proof}  The containment $\G_0 \leq \tG \cap \PSL_3(\Fqt)$ holds since we constructed $\G_0$ as a subgroup of $\tG$ and showed above that $\G_0 \leq \PSL_3(\Fqt)$.  

Let $g \in \tG \cap \PSL_3(\Fqt)$.  Then since the action of $\PSL_3(\Fqt)$ on $\Delta$ is type-preserving, the vertex $gv_0$ has type $0$.  Now as $\G_0$ acts transitively on vertices of type $0$, there is a $g_0 \in \G_0$ such that $g_0^{-1}gv_0 = v_0$.  Thus as $\G_0 \leq \tG$, the element $h := g_0^{-1}g$ is in $\Stab_{\tG}(v_0)$.  But $\Stab_{\G_0}(v_0) = \Stab_{\tG}(v_0) = N_0$, and thus $g = g_0 h \in \G_0$, as required.
\end{proof}

\subsubsection{Case $3 \mid q$}\label{s:3 divides q}

Note that in this case $(d,q-1) = 1$ and $p = 3$, so in particular $p$ divides $d=3$.

In this case, as $(3,q-1) = 1$, we have by Proposition \ref{p:Dm3} and Lemma~\ref{l:H}(4)  above that $H \cap \PSL_3(\Fqt)$ is equal to the Singer cycle $S < H$.  By similar arguments to those in Section \ref{s:3 divides q+1} above, it follows that $\Lambda'_0 = \G'_0$ is a cocompact lattice in $\PSL_3(\Fqt)$ with action as described in Corollary \ref{c:action} above.  The proof of the following lemma is similar to that of Lemma \ref{l:3 divides q+1} above.

\begin{lemma}\label{l:3 divides q} If $3$ divides $q$, then $\G'_0 = \tG \cap \PSL_3(\Fqt)$.
\end{lemma}

\subsubsection{Case $3 \mid (q - 1)$}\label{s:3 divides q-1}

In this case, by Lemma~\ref{l:H}(4) above, $H \cap \PSL_3(\Fqt)$ has order $(q^2 + q + 1)$.  Moreover, 
as $H\cap \PSL_3(\Fqt)=H\cap\PSL_3(\Fq[[t]])$,  $H$ is a normaliser of a Singer cycle of $\PSL_3(q)$.
  Thus as discussed in Section \ref{s:singer}, $H \cap \PSL_3(\Fqt)$ cannot act transitively on the set of points and the set of lines of the projective plane over $\F_q$.   Hence the arguments used to prove Proposition \ref{p:vertex transitive} above cannot be applied.  We do not know in this case whether $\Lambda_0$ or $\Lambda'_0$ acts transitively on the set of vertices of $\Delta$ of each type.  (Since $\Lambda_0$ and $\Lambda_0'$ are type-preserving cocompact lattices, we do know that they have finitely many orbits of vertices of each type.)

\section{Lattices in cases $d > 3$}\label{s:d>3}

As in the case $d = 3$, we first construct and establish the properties of lattices $\G'_0$ and $\G_0$ in $\PGL_d(\Fqt)$, then consider their intersections with $\PSL_d(\Fqt)$.  Many arguments from the case $d = 3$ apply immediately for $d > 3$.

\subsection{Lattices in $\PGL_d(\Fqt)$}\label{s:PGL_d}

For $d > 3$, the construction of the cocompact lattice $\tG$ in $\PGL_d(\Fqt)$ described in Section~\ref{s:CS} above appears in \cite{CS}.  
As recalled in Section \ref{s:applications}(5) above, the lattice $\tG$ is a product of a vertex stabiliser $H$ of order $d \frac{q^d - 1}{q - 1}$ and a vertex-regular lattice~$\G$.  Denote by $\tG'$ the subgroup of $\tG$ which is the product of $\G$ with the Singer cycle $S < H$ guaranteed by Lemma \ref{l:H} above.  Then by construction, $S$ is a vertex stabiliser in $\tG'$.

For $i = 0,\ldots,d-1$ let $v_i$ be the vertex of type $i$ in the standard chamber, as in Section \ref{s:buildings} above.  Let $N_i$ be the stabiliser of $v_i$ in $\tG$ and $S_i$  be the stabiliser of $v_i$ in $\tG'$.  Then each $N_i \cong H$ and each $S_i \cong S$.  We define 
\[ \G'_0 := \la S_0, \ldots, S_{d-1} \ra \leq \tG' \] and 
\[ \G_0 := \la N_0, \ldots, N_{d-1} \ra \leq \tG. \]  Clearly $\G'_0 \leq \G_0$.  

We claim that $\G'_0$ and $\G_0$ are cocompact lattices in $\PGL_d(\Fqt)$.  As in the case $d = 3$, it suffices to show that $\G_0$ is a discrete subgroup of $\PGL_d(\Fqt)$ and that $\G'_0$ acts cocompactly on $\Delta$, and the following lemma is immediate.

\begin{lemma} $\G_0$ is a discrete subgroup of $\PGL_d(\Fqt)$.
\end{lemma}

The proof of the next lemma is the same as that of Lemma \ref{l:neighbours} above, after replacing $3$ by $d$.  

\begin{lemma}\label{l:neighbours d>3}  For $i = 0,\ldots,d-1$ and $j = i-1, i+1 \pmod d$, the group $S_i$ acts simply transitively on the vertices neighbouring $v_i$ of type $j$.
\end{lemma}

Compared with the proof of the corresponding result in the case $d = 3$, Proposition \ref{p:vertex transitive} above, the proof of Proposition \ref{p:vertex transitive d>3} below requires some extra care in the base case of the induction.

\begin{prop}\label{p:vertex transitive d>3} For $i = 0,\ldots,d-1$, the group $\G'_0$ acts transitively on the vertices of type $i$ in $\Delta$.
\end{prop}

\begin{proof}  We will show that $\G'_0$ acts transitively on the vertices of type $0$ in $\Delta$.  The same argument will apply for types $i=1,\ldots,d-1$.  It suffices to show that for each vertex $w_0$ of type $0$, there is an element of $\G'_0$ which takes $w_0$ to $v_0$.  We prove this by induction on the distance $\delta(w_0,v_0) \in 2\mathbb{N}$.  

If $\delta(w_0,v_0) = 2$ we consider the following cases.
\begin{enumerate}
\item $w_0$ is adjacent to $v_1$.   By Lemma \ref{l:neighbours d>3} above, $S_1$ acts transitively on the type $0$ neighbours of $v_1$, and so the claim follows in this case.  \item $w_0$ is adjacent to some vertex $s_0 v_1$ with $s_0 \in S_0$.  Then $s_0^{-1} w_0$ is adjacent to $v_1$, and we apply the argument from Case (1).  \item $w_0$ is adjacent to $v_i$ where $i \in \{2,\ldots,d-1\}$.  Then there is a vertex $v_{i-1}'$ of type $(i - 1)$ so that $v_i$, $w_0$ and $v_{i-1}'$ are mutually adjacent. Since $S_i$ acts transitively on the type $(i-1)$ neighbours of $v_i$, we have that $s_i v_{i-1}' = v_{i-1}$ for some $s_i \in S_i$.  Thus $s_i w_0$ is adjacent to $v_{i-1}$.  By repeating this argument, we obtain after finitely many steps that for some $\gamma \in \G_0$ we have $\gamma w_0$ adjacent to $v_1$, and we may then apply the argument from Case (1).  \item $w_0$ is adjacent to a vertex $v_i' \neq v_i$ of type $i \in \{ 2, \ldots, d-1\}$, with $\delta(v_0, v_i') = \delta(v_i',w_0) = 1$.  Choose a vertex $v_1'$ of type $1$ so that $v_0$, $v_1'$ and $v_i'$ are mutually adjacent.  Then there is an $s_0 \in S_0$ such that $s_0 v_1' = v_1$, and hence $s_0 v_i'$ is a neighbour of $v_1$ of type $i$.  Now choose a vertex $v_2'$ of type $2$ so that $v_1$, $v_2'$ and $s_0v_i'$ are mutually adjacent.  Then there is an $s_1 \in S_1$ such that $s_1 v_2' = v_2$, and hence $s_1 s_0 v_i'$ is a neighbour of $v_2$ of type $i$.  By repeating this argument, we obtain that $\gamma v_i'$ is a neighbour of $v_{i-1}$ of type $i$, for some $\gamma \in \G'_0$.  Then there is an $s_{i-1} \in S_{i-1}$ such that $s_{i-1} \gamma v_i' = v_i$.  Thus $s_{i-1} \gamma w_0$ is a neighbour of $v_i$, and so we may apply the argument from Case (3).\end{enumerate}

The inductive step is exactly as in the case $d = 3$.
\end{proof}

\begin{corollary}\label{c:cocompact d>3} $\G'_0$ acts cocompactly on $\Delta$.
\end{corollary}

\begin{proof}  As in the case $d = 3$ (Corollary \ref{c:cocompact} above), this follows from the fact that $\G_0$ acts on $\Delta$ with finitely many orbits of vertices.
\end{proof}

We have established the claim that $\G'_0$ and $\G_0$ are cocompact lattices in $\PGL_d(\Fqt)$.  To finish the proof of Theorem \ref{t:PGL} in the case $d > 3$, we further describe the actions of $\G'_0$ and $\G_0$ on $\Delta$.  The proof of the following result is the same as for Corollary \ref{c:action} above.

\begin{corollary}\label{c:action d>3}  The action of $\G'_0$ and of $\G_0$ is type-preserving and transitive on each type of vertex in $\Delta$.  For $i = 0,\ldots,d-1$, the stabiliser of $v_i$ in $\G_0'$ is the group $S_i$, and in $\G_0$ is the group $N_i$.   
\end{corollary}

\subsection{Lattices in $\PSL_d(\Fqt)$}\label{s:PSL_d}

The proof in Section \ref{s:PSL_3} above that when $d = 3$ the groups $$\Lambda_0:=\G_0 \cap \PSL_d(\Fqt) \quad \mbox{and} \quad \Lambda'_0 :=\G'_0 \cap \PSL_d(\Fqt)$$ are cocompact lattices in $\PSL_3(\Fqt)$ generalises immediately to the cases $d \geq 3$.  However, describing these intersections becomes a bit more complicated, due to the various numerical possibilities.  We list the outcomes for various pairs of $d$ and $q$ in the next statement, which follows from Proposition \ref{p:Dm3} and Lemma \ref{l:H} above.  Recall that $S_i$ is a Singer cycle of $\PGL_d(q)$, hence $S_i \cong C_{\frac{q^d - 1}{q-1}}$, and that $N_i \cong C_{\frac{q^d -1}{q-1}} \rtimes C_d$. 
 
\begin{lemma}
\label{l:all lattices}
Let $q=p^a$, $a \in \N$, $d \geq 3$, and $i \in \{0,\ldots, d-1\}$.  
\begin{enumerate}
\item  Suppose that $(d,q-1) = 1$.
\begin{enumerate}
\item 
If $p$ does not divide $d$, then $$N_i \cap \PSL_d(\Fqt) \cong C_{\frac{q^d-1}{q-1}}\rtimes C_d$$ is equal to $N_i$.  Hence $\Lambda_0' = \G_0'$ and $\Lambda_0=\Gamma_0$.
\item If $p$ divides $d$, then  $$N_i\cap\PSL_d(\Fqt) \cong C_{\frac{q^d-1}{q-1}}\rtimes C_{\frac{d}{\Ord_p(d)}} $$  is a proper subgroup of $N_i$.  
Moreover, $S_i\leq  N_i\cap\PSL_d(\Fqt)$.  Hence $\Gamma_0'=\Lambda_0'$ and $\Lambda_0$ is a proper subgroup of $\Gamma_0$.
\end{enumerate}
\item Suppose that $(d,q-1) \neq 1$.
\begin{enumerate}
\item  If $p$ does not divide $d$, then  $$N_i\cap\PSL_d(\Fqt) \cong C_{\frac{q^d-1}{(q-1)(d,q-1)}}\rtimes C_d$$ is a proper subgroup of $N_i$.  Moreover, $S_i$ is not contained in $N_i\cap\PSL_d(\Fqt)$.  Hence $\Gamma_0'$ is a proper subgroup of $\Lambda_0'$ and $\Lambda_0$ is a proper subgroup of $\Gamma_0$.\item If $p$ divides $d$, then $$N_i\cap\PSL_d(\Fqt) \cong C_{\frac{q^d-1}{(q-1)(d,q-1)}}\rtimes C_{\frac{d}{\Ord_p(d)}}$$ is a proper subgroup of $N_i$. Moreover, $S_i$ is not contained in $N_i\cap\PSL_d(\Fqt)$.  Hence $\Gamma_0'$ is a proper subgroup of $\Lambda_0'$ and $\Lambda_0$ is a proper subgroup of $\Gamma_0$.

\end{enumerate}
\end{enumerate}
\end{lemma}

The following relationships between the lattices $\G_0$ and $\G_0'$ and the Cartwright--Steger lattice $\tG$ are implied by Lemma \ref{l:all lattices} above, together with similar arguments to those used in Lemmas \ref{l:3 divides q+1} and \ref{l:3 divides q} above.

\begin{lemma} Assume that $(d,q-1) = 1$.  If $p$ does not divide $d$, then $\G_0 = \tG \cap \PSL_d(\Fqt)$, while if $p$ divides $d$, then $\G'_0 \leq \tG \cap \PSL_d(\Fqt)$.
\end{lemma}

\section{Minimality of covolumes}\label{s:covolumes}

In Section \ref{s:no p elements} we discuss whether cocompact lattices in the matrix groups we have been considering can contain $p$--elements.  We then in Section \ref{s:minimality} discuss minimality of covolumes of cocompact lattices in $G = \SL_3(\Fqt)$.

\subsection{Cocompact lattices, do they contain $p$--elements?}\label{s:no p elements}

We begin by establishing an analogue for $G = \SL_d(\Fqt)$ of Godement's Cocompactness Criterion.  This result, which was proved  by Borel and Harish-Chandra \cite{BHC} and independently by Mostow--Tamagawa \cite{MT}, states that for $G$ a semisimple $\mathbb{Q}$--algebraic group and $\G$ a lattice in $G$, $\G$ is cocompact if and only if $\G$ contains no non-trivial unipotent elements.  An element of $\GL(n,\mathbb{C})$ is \emph{unipotent} if all of its eigenvalues are equal to $1$.

We will use the general result contained in Proposition \ref{p:gelfand} below.  A similar statement can be found in, for example,~\cite[page 10]{GGPS}.  The proof in \cite{GGPS} requires a compact fundamental domain, that cannot be assured in our case.  Hence, for the sake of completeness, we exhibit a variation of their argument here.  The existence of a discrete cocompact subgroup will make the group $G$ locally compact, 
but we still formulate the result for a topological group because local compactness is not used in the proof.

\begin{prop}
\label{p:gelfand} 
Let $G$ be a topological group and $\Gamma$ a discrete cocompact subgroup of $G$.  
If $u\in\Gamma$, then $$u^G := \{ gug^{-1} \mid g \in G\}$$ is a closed subset of $G$. 
\end{prop}

\begin{proof}
Let $g_i u g_i^{-1}$, $g_i\in G$, be a net converging to $v\in G$.
Since $\Gamma$ is cocompact, the set $\{ g_i\Gamma  \}$ admits
a convergent subnet, so without loss of generality, $g_i\Gamma \rightarrow g\Gamma $.
Thus, there exist such $x_i\in\Gamma$ that  $g_ix_i \rightarrow g$. 
Since $g_i u g_i^{-1} = (g_ix_i)( x_i^{-1} ux_i) (g_ix_i)^{-1}$,
the net $x_i^{-1} ux_i$ converges to $g^{-1}vg$.
Since all $x_i^{-1} ux_i$ are elements of the discrete subgroup $\Gamma$,
the net must stabilise, hence,  $x_j^{-1} ux_j=g^{-1}vg$  for some $j$,
and so we arrive at $v\in u^G$.
\end{proof}

It is an interesting question whether cocompact lattices in groups defined over a field of characteristic $p$ contain $p$--elements.  In \cite{L1} Lubotzky uses Proposition \ref{p:gelfand} above to show that cocompact lattices in $\SL_2(\Fqt)$, where $q=p^a$, contain no $p$--elements. In fact, this statement can be generalised in the following way.  

\begin{prop}\label{p:no p elements} Let $G = \SL_d(\Fqt)$ where $q = p^a$ with $p$ prime and $d\geq 2$.  Let $\G$ be a lattice in $G$.  Then $\G$ is cocompact if and only if $\G$ does not contain any elements of order $p$.
\end{prop}

\begin{proof}  First suppose that $\G$ is non-cocompact and let $A$ be a set of vertices of the building for $G$ which represent the orbits of $\G$.  Then by the remarks in the introduction, $A$ is infinite and the series $\mu(\G \bs G) = \sum_{a \in A} |\Stab_\G(a)|^{-1}$ converges, hence $\G$ contains vertex stabilisers of arbitrarily large order.  The Levi decomposition (Proposition \ref{p:levi} above) then implies that $\G$ must have elements of order $p$.

For the converse, by Proposition \ref{p:gelfand} above, it is enough to show that if $u \in G$ is a $p$--element then there is $g \in G$ such that $g^k u g^{-k} \to I$ as $k \to \infty$, where $I$ is the identity matrix in $G$.

So let $u \in G$ be such that $u^p = I \neq u$.  Since we are working over a field of characteristic $p$, it follows that $(u - I)^p = 0$ and thus $u$ is a unipotent element of $G=\SL_d(\Fqt)$ (recall that by definition, unipotent elements are those with all eigenvalues equal to $1$).  Thus $u$ is conjugate in $G$ to a matrix with all $1$s on the diagonal and all below-diagonal elements $0$.  Without loss of generality we may assume that $u$ itself has all $1$s on the diagonal and all below-diagonal elements $0$. It is then not hard to construct a suitable diagonal matrix $g \in G$ such that $g^k u g^{-k}$ converges to $I$.
For example, for $d=3$, $g$ can be taken to be the following matrix:
$$\begin{pmatrix}
t^{2}&0&0\\
0&t&0\\
0&0&t^{-3}\end{pmatrix}.$$
\end{proof}

The proof of Proposition \ref{p:no p elements} makes essential use of the fact that in $\SL_d(\Fqt)$, an element of order $p$ is a genuine unipotent element (that is, is conjugate of a matrix with eigenvalues $1$).  However,  {\em one needs to be careful about cocompact lattices in other matrix groups}!

 Let us look again at the Cartwright--Steger lattice $\widetilde{\G}$ in $\PGL_d(\Fqt)$.  As we saw, $\widetilde{\G}=\G H$ where $H$ is a finite subgroup of $\PGL_d(\Fqt)$ of order $d\frac{(q^d-1)}{(q-1)}$.  Suppose that $p$ divides $d$ (for example, if $p=3=d$). Then obviously $H$, and thus $\widetilde{\G}$,  contains an element $\widetilde{h}\in H$ of order $p$.  On the other hand, $\widetilde{\G}$ is a cocompact lattice in $\PGL_d(\Fqt)$. What is going on?  The answer comes from the fact that under the natural map $\GL_d(\Fqt)\rightarrow \PGL_d(\Fqt)$, $\widetilde{h}$ is the image of an element $h\in \GL_d(\Fqt)$ of infinite order. Hence, $\widetilde{h}$ is not  ``genuinely unipotent" and the proof of Proposition~\ref{p:no p elements} above does not work.
In fact the conjugacy class of $\widetilde{h}$ in $\PGL_d(\Fqt)$ is closed, so there is no contradiction with Proposition \ref{p:gelfand} above.

\subsection{Minimality of covolumes}\label{s:minimality}

As discussed in the introduction, our original motivation was to find cocompact lattices of minimal covolume in $\SL_3(\Fqt)$, and this led us to considering vertex stabilisers which are Singer cycles or normalisers of Singer cycles.  We now consider covolumes of cocompact lattices in the special case that $G = \SL_3(\Fqt)$ and $(3,q-1) = 1$. 
Notice that in particular, $\SL_3(\Fqt) = \PSL_3(\Fqt)$.  

By Theorem \ref{t:PSL} and the remarks in the introduction, we have that $\G'_0$ is a cocompact lattice in $G$ of covolume
\[ \mu(\G'_0 \bs G) = \sum_{i=0}^2 \frac{1}{|\Stab_{\G'_0}(v_i)|}  = \sum_{i=0}^2 \frac{1}{|S_i|}= \frac{3}{q^2 + q + 1}.\]
Also, if $p \neq 3$, then $\G_0$ is a cocompact lattice in $G$ of covolume
\[ \mu(\G_0 \bs G) = \sum_{i=0}^2 \frac{1}{|\Stab_{\G_0}(v_i)|}  = \sum_{i=0}^2 \frac{1}{|N_i|}= \frac{3}{3(q^2 + q + 1)} = \frac{1}{q^2 + q + 1}.\]

Now let $\G$ be any cocompact lattice in $G = \SL_3(\Fqt)$.  Then by Proposition \ref{p:no p elements} above, each vertex stabiliser in $\G$ is a finite $p'$--subgroup of a vertex stabiliser in $G$.  The Levi decomposition (Proposition \ref{p:levi} above) then implies that each vertex stabiliser in $\G$ is isomorphic to a $p'$--subgroup of $\SL_3(q) = \PSL_3(q)$.  We thus consider maximal $p'$--subgroups of $\PSL_3(q)$, in Lemma \ref{l:maximal p'} below.  

Note that since $\G$ is type-preserving, $\G$ has at least one orbit of vertices of each type $i = 0,1,2$.  It follows that if $|\Stab_\G(v_i)| \leq q^2$ for each $i$, then $\mu(\G \bs G) > \mu(\G'_0 \bs G)$ and so $\G$ is not a cocompact lattice of minimal covolume.  Hence in the next statement we consider only maximal $p'$--subgroups of order greater than $q^2$. 

\begin{lemma}\label{l:maximal p'}
Let $K=\PSL_3(q)$, where $q=p^a>72$ with $p$ prime and $a\in \mN$.  Assume that $(3,q-1) = 1$ and $q > 72$.   Let $H$ be a maximal  $p'$--subgroup of $K$ with $|H|>q^2$.  Then one of the following conditions holds.

If $p = 2$:
\begin{enumerate}
\item  $H$ is a subgroup of the normaliser of a maximal split torus of $K$ and $|H|=3(q-1)^2$; or
\item $H$ is the normaliser of a Singer cycle of $K$ and $|H|=3(q^2+q+1)$.
\end{enumerate}

If $p=3$:
\begin{enumerate}
\item $H$ is a subgroup of the normaliser of a maximal split torus of $K$ and $|H|=2(q-1)^2$;
\item  $H$ is the normaliser of a Singer cycle of $K$ and $|H|=(q^2+q+1)$; or
\item $H$ is a subgroup of a Levi complement of a maximal parabolic subgroup of $K$ and $|H|=2(q^2-1)$.
\end{enumerate}

If $p\geq 5$:
\begin{enumerate}
\item $H$ is  the normaliser of a maximal split torus of $K$ and $|H|=6(q-1)^2$;
\item  $H$ is the normaliser of a Singer cycle of $K$ and $|H|=3(q^2+q+1)$; or
\item $H$ is a subgroup of a Levi complement of a maximal parabolic subgroup of $K$ and $|H|=2(q^2-1)$.
\end{enumerate}

\end{lemma}

\begin{proof}  The result follows immediately from the theorem of Hartley and Mitchell (cf. Theorem  6.5.3 of \cite{GLS3}).
\end{proof}

From this, the following minimality result in characteristic $2$ is immediate:

\begin{prop}\label{p:minimality p=2 d=3}  Suppose that $(3,q-1) = 1$ and that $p = 2$.  Then for $q$ large enough, the lattice $\G_0$ is a cocompact lattice of minimal covolume in $G=\SL_3(\Fqt)$.  \end{prop}

\begin{proof}  Let $\G$ be any cocompact lattice in $\SL_3(\Fqt)$ and assume that $q > 72$.   By Lemma \ref{l:maximal p'} and the discussion preceding it, for $i = 0,1,2$, we have $|\Stab_\G(v_i)| \leq |\Stab_{\G_0}(v_i)| = 3(q^2 + q + 1)$ and so $\mu(\G \bs G) \geq \mu(\G_0 \bs G)$ as required.\end{proof}

It would be nice either to prove or to disprove Proposition~\ref{p:minimality p=2 d=3} in an arbitrary characteristic $p$.  At the moment of writing, we cannot do it, for reasons we now explain.

A lattice $\G' \leq G = \SL_3(\Fqt)$ is said to be \emph{maximal} if for every lattice $\G \leq G$ such that $\G' \leq \G$, in fact $\G' = \G$.  It is clear that a cocompact lattice of minimal covolume must be a maximal lattice.  In fact,  the following is true.

\begin{prop}\label{p:maximal}  Suppose that $(3,q-1)=1$. Then for $q$ large enough, if $p=3$, the lattice $\G'_0$ is a maximal lattice in $G=\SL_3(\Fqt)$ and if $p\geq 5$, the lattice $\G_0$ is a maximal lattice in $G=\SL_3(\Fqt)$. \end{prop}

\begin{proof}  We give the proof for $p\geq 5$.   The proof for $p = 3$ is similar.  Suppose that $\G$ is a lattice in $G$ such that $\G_0 \leq \G$.  Then $\G$ is cocompact, since $\G_0$ is cocompact.  
 Since $\G$ is type-preserving and $\G_0$ is transitive on each type of vertex, $\G$ is transitive on each type of vertex.  By Lemma \ref{l:maximal p'}, the vertex stabilisers in $\G_0$ are maximal $p'$--subgroups of $\PSL_3(q)$.  It follows that for $i = 0,1,2$ we have $\Stab_\G(v_i) = \Stab_{\G_0}(v_i)$ and hence $\mu(\G \bs G) = \mu(\G_0 \bs G)$.  Thus $\G = \G_0$ as required. 
\end{proof}
 

For $p \geq 5$, we have found a candidate besides $\G_0$ for the cocompact lattice of minimal covolume.  Let $H_1$ be the normaliser of a maximal split torus of $\PSL_3(q)$.  
Using complexes of groups (see \cite{BrH}), for $p$ odd and $(3,q-1) = 1$ we are able to construct a group $\G_1$ which acts transitively on the set of vertices of each type in \emph{some} building of type $\tilde{A}_2$ (possibly exotic), so that each vertex stabiliser in $\G_1$ is isomorphic to $H_1$.  However, for $p \geq 5$ we do not know whether $\G_1$ embeds in $G = \SL_3(\Fqt)$ as a cocompact lattice acting transitively on the set of vertices of each type in the building for $G$, with $\Stab_{\G_1}(v_i) \cong H_1$ for $i = 0,1,2$.  (For $p=3$, the whole group $H_1$ cannot be a vertex stabiliser, since it contains an element of order $3$.)
If there is such an embedding of $\G_1$, then by the same arguments as for Proposition \ref{p:maximal}, $\G_1$ is a maximal lattice in $G$, and it will have a smaller covolume than $\Gamma_0$:
\begin{eqnarray*}\mu(\G_1 \bs G)  & = & \sum_{i=0}^2 \frac{1}{|\Stab_{\G_1}(v_i)|}  = \sum_{i=0}^2 \frac{1}{|6(q-1)^2|}\\ & = & \frac{3}{6(q-1)^2} = \frac{1}{2(q-1)^2} < \frac{1}{q^2 + q + 1}.\end{eqnarray*}

Hence, we would like to finish this section with the following question and conjecture.
\vskip 3mm
\noindent {\bf Question.} Does $G = \SL_3(\Fqt)$ admit a lattice $\Gamma_1$ as described above?
\vskip 3mm
\noindent {\bf Conjecture.} Let $(p,3)=1=(3,q-1)$ and $G=\SL_3(\Fqt)$. Then either $\G_0$ is a cocompact lattice of minimal covolume, or $G$ admits a cocompact lattice $\Gamma_1$ as described above, and $\Gamma_1$ is a cocompact lattice of minimal covolume.

\section{Relationship with the work of Essert}\label{s:essert}

Recall from the introduction that Essert \cite{E} constructed cocompact lattices which act simply transitively on the set of panels of the same type in some $\tilde{A}_2$--building, possibly exotic.  We now conclude by resolving some open questions from \cite{E}.  

To explain these questions, let $\Delta$ be the building $\tilde{A}_2(K,\nu)$, for some field $K$ with discrete valuation~$\nu$, and let $G = \cG(K)$ where $\cG$ is in the set $\{ \PGL_3, \SL_3, \PSL_3\}$.  Suppose that $\G$ is a cocompact lattice in $\Aut(\Delta)$, meaning that $\G$ acts cocompactly on $\Delta$ with finite stabilisers.  
Since $G/Z(G)$ is not equal to $\Aut(\Delta)$, 
it is possible that $\G$ is not contained in $G$ even though $\G$ acts on the building associated to $G$. On the other hand, since $G/Z(G)$ is cocompact in $\Aut(\Delta)$, if $\G$ is a cocompact lattice in $G$, then $\G$ will be a cocompact lattice in $\Aut(\Delta)$.  The Mostow--Margulis Rigidity Theorem (see \cite{M}) implies that the group $\G$ cannot be a lattice in $\cG(K)$ for two different fields $K$.  

With the exception of one lattice which is realised explicitly in the group $\SL_3(\F_2(\!(t)\!)\!)$ (see the Remark in \cite[Section 5.2]{E}), it is an open question in \cite{E} whether the lattices constructed there act on any building $\tilde{A}_2(K,\nu)$, and also whether they can be embedded in any $\cG(K)$.  We consider these questions in the case that $K = \Fqt\,$.

Let $\Delta = \tilde{A}_2(\Fqt\,,\nu)$.  We first consider the lattice $\G_0' \leq \PGL_3(\Fqt)$ constructed in Section~\ref{s:PGL_3} above.  Since the vertex stabilisers of $\G'_0$ are Singer cycles of $\PGL_3(q)$, and $\G_0'$ acts transitively on the set of vertices of each type in $\Delta$, it follows that the lattice $\G_0'$ acts simply transitively on the set of panels of each type in $\Delta$.  Thus the lattice $\G'_0$ is of the form considered by Essert \cite{E}, and is contained in $\PGL_3(\Fqt)$ for all $q$.  From the discussion above, it follows that for all $q$, there is a lattice in $\Aut(\Delta)$ acting simply transitively on the set of panels of the same type.

Next suppose that $(3,q-1) = 1$.  We showed in Section \ref{s:PSL_3} above that in this case, the lattice $\G_0'$ is also contained in $\PSL_3(\Fqt) = \SL_3(\Fqt)$.  Hence for all $q$ such that $(3,q-1) = 1$, there is a lattice in $\SL_3(\Fqt)$ which acts simply transitively on the set of panels of the same type.

Finally suppose that $3 \mid (q-1)$.  From the Levi decomposition (Proposition \ref{p:levi} above) and Proposition \ref{p:no p elements} above, if $\G$ is a cocompact lattice in $\SL_3(\Fqt)$, then the vertex stabilisers in $\G$ are isomorphic to $p'$--subgroups of $\SL_3(q)$.  However, when $q$ is large enough and $3 \mid (q-1)$, there is no $p'$--subgroup of $\SL_3(q)$ which acts transitively on the points of the projective plane (see Section \ref{s:singer}).  Hence no vertex stabiliser in $\G$ can act transitively on the set of adjacent panels of the same type.  Thus if $q$ is large enough and $3 \mid (q-1)$, there is no lattice $\G < \SL_3(\Fqt)$ which acts (simply) transitively on the set of panels of the same type.

\end{document}